\UseRawInputEncoding

\documentclass[12pt]{amsart}
\usepackage{amsfonts}
\usepackage{amsmath}
\usepackage{amssymb,latexsym}
\usepackage[mathcal]{eucal}
\usepackage[pdftex,bookmarks,colorlinks,breaklinks]{hyperref}

\input xy
\xyoption{all}

\newcommand{\Bcal}{\mathcal{B}}

\newcommand{\Hcal}{\mathcal{H}}

\newcommand{\Lcal}{\mathcal{L}}

\newcommand{\Xcal}{\mathcal{X}}
\newcommand{\Ycal}{\mathcal{Y}}

\newcommand{\Q}{\mathbb{Q}}

\newcommand{\C}{\mathbb{C}}
\newcommand{\N}{\mathbb{N}}

\newcommand{\del}{\delta}

\newcommand{\sig}{\sigma}

\newcommand{\la}{\lambda}

\newcommand{\tet}{\theta}

\newcommand{\om}{\omega}
\newcommand{\Om}{\Omega}

\newcommand{\ol}{\overline}

\newcommand{\br}{\vspace{3 mm}}
\newcommand{\imp}{\Rightarrow}

\newcommand{\ext}{{\rm{ext\,}}}

\newcommand{\supp}{{\rm{supp\,}}}

\newcommand{\Homeo}{{\rm Homeo}}

\swapnumbers
\theoremstyle{plain}
\newtheorem{thm}{Theorem}[section]
\newtheorem{cor}[thm]{Corollary}

\theoremstyle{definition}
\newtheorem{defn}[thm]{Definition}

\newtheorem{rmk}[thm]{Remark}

\begin{document}

%%%%%%%%%%%%%%%%%%%%%%%%%%%%%%%%%%%%%%%%%%%%%%%%%%%%%%%%%%%%%%%%%%%%
%%%%%%%%%%%%%%%%%%%   T H E    T I T L E    %%%%%%%%%%%%%%%%%%%%%%%%%%
%%%%%%%%%%%%%%%%%%%%%%%%%%%%%%%%%%%%%%%%%%%%%%%%%%%%%%%%%%%%%%%%%%%%

%\title[Positive entropy systems are dominant]
%{Positive entropy systems are dominant}
\title[The Unique stationary boundary property]
{The Unique stationary boundary property}

\author{Eli Glasner}

\address{Department of Mathematics\\
     Tel Aviv University\\
         Tel Aviv\\
         Israel}
\email{glasner@math.tau.ac.il}

%%%%%%%%%%%%%%%%%%%%%%%%%%%%%%%%%%%%%%%%%%%%%%%%%%%%%%%%%%%%%%%%%%%%%
%%%%%%%%%%%%%%%%%%%%%%%%  CONTENTS   %%%%%%%%%%%%%%%%%%%%%%%%%%%%%%%%
%%%%%%%%%%%%%%%%%%%%%%%%%%%%%%%%%%%%%%%%%%%%%%%%%%%%%%%%%%%%%%%%%%%%%

%\tableofcontents
\setcounter{secnumdepth}{2}

%\addtocontents{toc}{subsection}{\protect\hspace{0.5cm}}

%%%%%%%%%%%%%%%%%%%%%%%%%%%%%%%%%%%%%%%%%%%%%%%%%%%%%%%%%%%%%%%%%%%%%
%%%%%%%%%%%%%%%%%%%%%%%%%%%%  TEXT   %%%%%%%%%%%%%%%%%%%%%%%%%%%%%%%%
%%%%%%%%%%%%%%%%%%%%%%%%%%%%%%%%%%%%%%%%%%%%%%%%%%%%%%%%%%%%%%%%%%%%%

\setcounter{section}{0}
%\setcounter{page}{0}

%%%%%%%%%%%%%%%%%%%%%%%%%%%%%%%%%%%%%%%%%%%%%%%%%%%%%%%%%%%%%%%%%%%%%
%%%%%%%%%%%%%%%%%%%       Inroduction         %%%%%%%%%%%%%%%%%%%%%%%
%%%%%%%%%%%%%%%%%%%%%%%%%%%%%%%%%%%%%%%%%%%%%%%%%%%%%%%%%%%%%%%%%%%%%

%\subjclass[2010]{Primary 37A15, 37A35, 37B05}
%
%\keywords{distal systems, structure theory, cocycles}

%\begin{abstract}
%
%\end{abstract}
%
%\keywords{ }
%
%\thanks{The first named author was supported by grant \# 1194/19  of the Israel Science Foundation.}

\subjclass[2010]{Primary 22D40, 37A50, 37B05}
\keywords{$\mu$-stationary measures, Poisson boundary, strong proximality}

\begin{abstract}
%Let $G$ be a locally compact group and $\mu$ an admissible probability measure on $G$.
%Let $(B,\nu)$ be the universal topological Poisson $\mu$-boundary of $(G,\mu)$ and  
%$\Pi_s(G)$ the universal minimal strongly proximal $G$-flow.
%Extending a recent result of Hartman and Kalantar we show that
%the following conditions are equivalent 
%(i) the measure $\nu$ is the unique $\mu$-stationary measure on $B$,
%(ii) $(B,G) \cong (\Pi_s(G),G)$. 
Let $G$ be a locally compact group and $\mu$ an admissible probability measure on $G$.
Let $(B,\nu)$ be the universal topological Poisson $\mu$-boundary of $(G,\mu)$ and  
$\Pi_s(G)$ the universal minimal strongly proximal $G$-flow.
%In this note we present an alternative proof to a 
%recent result of Hartman and Kalantar.
This note is inspired by a recent result of Hartman and Kalantar.
 We show that for a locally compact second countable group $G$
the following conditions are equivalent: 
(i) the measure $\nu$ is the unique $\mu$-stationary measure on $B$,
(ii) $(B,G) \cong (\Pi_s(G),G)$. 
\end{abstract}
%
%\keywords{ }
%
%\thanks{The first named author was supported by grant \# 1194/19  of the Israel Science Foundation.}

\begin{date}
{May 1, 2023}
\end{date}

\maketitle
%%%%%%%%%%%%%%%%%%%%%%%%%%%%%%%%%%%%%%%%%%%%%%%%%%%%%%%%%%%%%%%%%%%%%
%%%%%%%%%%%%%%%%%%%%%%%%  CONTENTS   %%%%%%%%%%%%%%%%%%%%%%%%%%%%%%%%
%%%%%%%%%%%%%%%%%%%%%%%%%%%%%%%%%%%%%%%%%%%%%%%%%%%%%%%%%%%%%%%%%%%%%

%\tableofcontents
\setcounter{secnumdepth}{2}

%\addtocontents{toc}{subsection}{\protect\hspace{0.5cm}}

\setcounter{section}{0}
%\setcounter{page}{0}

%%%%%%%%%%%%%%%%%%%%%%%%%%%%%%%%%%%%%%%%%%%%%%%%%%%%%%%%%%%%%%%%%%%%%
%%%%%%%%%%%%%%%%%%%       Inroduction         %%%%%%%%%%%%%%%%%%%%%%%
%%%%%%%%%%%%%%%%%%%%%%%%%%%%%%%%%%%%%%%%%%%%%%%%%%%%%%%%%%%%%%%%%%%%%

\section{Introduction}

This is a short note with a long introduction. Those readers who are familiar with the theory
of stationary dynamical systems, as presented e.g. in \cite{Fu-71},  \cite{Gl-76},  \cite{FG-10} and \cite{HK-23} 
can go directly to Section \ref{sec:main} and come back to the introduction when needed.

\subsection{Flows}
Let $G$ be a locally compact second countable group (LCSC for short).
A {\em flow} is a pair $(X,G)$ where $X$ is a compact Hausdorff space and $G$ acts on $X$ via
a continuous homomorphism of $G$ into the Polish group $\Homeo(X)$ of self-homeomorphisms
of $X$ equipped with the  topology of uniform convergence.

A flow $(X,G)$ is {\bf minimal} if every orbit $Gx$ is dense in $X$. A pair of points $x, x' \in X$ 
is {\bf proximal} if there is a net $g_i \in G$ and a point $z  \in X$ such that $\lim_{i}g_ix = \lim_{i} g_i x' =z$.
The flow is {\bf proximal} is every pair $(x,x') \in X \times X$ is proximal.
A flow $(X,G)$ is called {\bf {strongly proximal}}  
if  for every Borel probability measure $\nu$ on $X$, there is a net $g_i \in G$
and a point $x \in X$ such that 
$$
w^*{\text{-}}\lim (g_i)_*(\nu) \to \delta_x.
$$
The group $G$ admits a {\bf universal minimal strongly proximal flow},
denoted by $(\Pi_s(G),G)$. It is unique in the sense that the identity map
is the only endomorphism it admits (see \cite[Chapter III]{Gl-76}).
Thus, every minimal strongly proximal
flow  is a factor of the system $(\Pi_s(G),G)$.

Recall that a {\bf homomorphism} (or an {\bf extension} or a{\bf factor map}) of flows $\pi :(X,G) \to (Y,G)$
 is a continuous surjection that intertwines the actions on $X$ and $Y$. 
 We say that the extension $\pi$ is {\bf strongly proximal}
if for every $y \in Y$ and every probability measure $\rho$ supported on $\pi^{-1}(y)$,
there is a net $g_i \in G$ such that the net $g_i \rho$ converges to a point mass. 
The extension $\pi$ has a {\bf relatively invariant measure} (RIM for short)  if
there is a {\bf continuous} map $y \mapsto \la_y \in  M(X)$, 
the compact space of probability measures on $X$,
which is equivariant (i.e. $g\la_y = \la_{gy}$ for every $g \in G$ and $y \in Y$) and
such that $\pi_*(\la_y) = \del_y$, for every $y \in Y$,
see \cite{Gl-75}.

We have the following:

\begin{thm}\label{RIM}\cite[Theorem 4.1]{Gl-75}
Let $\phi : (X,G) \to (Y,G)$ be a homomorphism of minimal flows.
Then there is a canonically defined commutative diagram of minimal flows
\begin{equation}\label{rim}
\xymatrix
{
X \ar[d]_{\phi}  & \tilde{X} = X \vee \tilde{Y} \ar[l]_-{\tilde{\tet}} \ar[d]^{\tilde{\phi}} \\
Y & \tilde{Y}\ar[l]^{\theta},
}
\end{equation}
where $\tet$ and $\tilde{\tet}$ are strongly proximal extensions, and $\tilde{\phi}$ has a RIM.
\end{thm}

\begin{defn}\label{contr}
 Let $(X,G)$ be a flow and $\nu$ a probability measure on $X$.
 We say that $\nu$ is {\bf contractible} if the set $\{\del_x : x \in X\}$ is contained in the orbit closure
 of $\nu$ in $M(X)$.
 \end{defn}

\begin{rmk}
In the sequel the word ``flow" will always mean a compact topological action, whereas
the word ``system" will refer to measure actions which may or may not be also flows.
It is well known and not hard to see that every, say quasi-invariant system $(X,\nu, G)$ admits a flow model, i.e. 
there is a flow $(Y,G)$ and a probability measure $\rho$ on $Y$ such that the systems
$(X,\nu, G)$ and $(Y,\rho,G)$ are measure isomorphic.
\end{rmk}

\br

\subsection{$\mu$-stationary dynamical systems and the measurable Poisson boundary}

The theory of $\mu$-harmonic functions and the
corresponding Poisson boundary was developed by Harry Furstenberg in several pioneering works
(see e.g. \cite{Fu-63}, \cite{Fu-71}, \cite{Fu-73}).

Let  $\mu$ be an {\bf admissible}  measure on $G$;
i.e. a probability measure with the following two properties: 
(i) For some $k \ge 1$ the convolution power 
$\mu^{*k}$ is absolutely continuous with respect to Haar measure,
(ii) the smallest closed subgroup containing $\supp(\mu)$
is all of $G$.

Let $(X,\Bcal)$ be a standard Borel space and let
$G$ act on it in a measurable way. A probability measure $\nu$ on
$X$ is called {\bf $\mu$-stationary}, or just stationary when $\mu$ is
understood, if $\mu*\nu=\nu$. 
As shown by Nevo and Zimmer \cite{NZ-99}, every $\mu$-stationary
probability measure $\nu$ on a $G$-space $X$ is quasi-invariant;
i.e. for every $g\in G$, $\nu$ and $g\nu$ have the same null sets.
Given a stationary measure $\nu$ the
triple  $\Xcal =(X,\nu,G)$ is called a {\bf
$\mu$-stationary dynamical system}, or just a $\mu$-{\bf system\/}. 
%(Usually we omit the
%$\sig$-algebra $\Bcal$ from the notation of an $m$-system, and
%often also the group $G$ and the measure $\mu$). 

\begin{rmk}\label{MK}
By the Markov-Kakutani fixed point theorem, every flow $(X,G)$ admits at least one
$\mu$-stationary probability measure.
\end{rmk}

\br

Let $\Omega=G^\N$ and let $P=\mu^\N=\mu\times \mu \times \mu \dots$ be the
product measure on $\Omega$. We let $\xi_n:\Om\to G$, denote the
projection onto the $n$-th coordinate,\ $\ n=1,2,\dots$. 
The stochastic process $(\Om,P,\{\eta_n\}_{n\in\N})$, where
$\eta_n=\xi_1\xi_2\cdots\xi_n$,  is called the {\bf $\mu$-random walk on $G$}.
By the martingale convergence theorem we have the existence for $P$
almost all $\om\in\Om$ of the limits
\begin{equation*}
\lim_{n\to\infty}\eta_n\nu=\nu_\om,
\end{equation*}
which are called the {\bf conditional measures} of the $\mu$-system $\Xcal $. 

The map $\zeta : \Om \to M(X)$ defined a.s. by $\om \mapsto \nu_\om$,
sends the measure $P$ onto a $\mu$-stationary probability measure $\zeta_*P = P^*$ on $M(X)$.
We denote the resulting $\mu$-systems by $(\Pi(\Xcal), P^*, G)$.
This is a {\bf quasifactor} of the system $\Xcal$ (for more details on quasifactors see e.g. \cite{Gl-03})
meaning that the measure $P^*$ satisfies the
{\bf barycenter} equation:
$$
\int \nu_\om \, d P(\om) = \nu.
$$

\br

\begin{defn}
Let $\pi: (X,\nu,G) \to (Y, \eta,G)$ be a homomorphism of
$\mu$-dynamical systems.
\begin{itemize}
\item
The homomorphism $\pi$ is  {\bf measure
preserving} (or extension) if for every $g\in G$ we
have $g\nu_y=\nu_{gy}$ for $\eta$  almost all $y$. Here the
probability measures $\nu_y\in M(X)$ are those given by the
disintegration $\nu=\int \nu_y d\eta(y)$.  
It is easy to see that
when $\pi$ is a measure preserving extension then also (with
obvious notations),  $P$ a.s. $g(\nu_\om)_y= (\nu_\om)_{gy}$
 for $\eta$  almost all $y$. 
 Clearly, when $\Ycal $ is the
trivial system, the extension $\pi$ is measure preserving iff the
system $\Xcal $ is measure preserving.
\item
The $\mu$-system $\Xcal =(X,G,\nu)$  is
{\bf $\mu$-proximal} (or a ``boundary" in the terminology of \cite{Fu-71})
if $P$ a.s. the conditional measures $\nu_\om\in M(X)$ are point
masses.  Note that for every $\mu$-system $\Xcal =(X,G,\nu)$ the corresponding 
quasifactor $(\Pi(\Xcal), P^*, G)$ is $\mu$-proximal.
\item
The homomorphism (or extension) $\pi$ is {\bf $\mu$-proximal}  if $P$ a.s. the
extension $\pi:(X,\nu_\om)\to (Y,\eta_\om)$ is a.s. one-to-one, where
$\eta_\om$ are the conditional measures for the system $\Ycal =(Y,\eta,G)$.
Clearly, when $\Ycal $ is the trivial system, the extension $\pi$
is $\mu$-proximal iff the system $\Xcal $ is $\mu$-proximal.
%When
%there is no room for confusion we sometimes say proximal rather
%than $\mu$-proximal.
\item
A $\mu$-system $(X,\nu)$  is called {\bf standard} if there exists
a homomorphism $\pi:(X,\nu)\to(Y,\eta)$, with $(Y,\eta)$ $\mu$-proximal and
the homomorphism $\pi$ a measure preserving extension.
\item
A bounded measurable complex valued function $f : G \to \C$ 
satisfying the equation $\mu *f =f$ is called  $\mu$-{\bf harmonic}.
\end{itemize}
\end{defn}

Given the group $G$ and the probability measure $\mu$, there exists a
unique (up to measurable isomorphisms)  universal $\mu$-proximal system $(\Pi(G,\mu),\eta)$ called the
{\bf universal measurable Poisson $\mu$-boundary} of the pair $(G,\mu)$. Thus, every $\mu$-proximal
system $(X,\mu)$ is a factor of the system $(\Pi(G,\mu),\eta)$.

\br 

It is easy to check that given a $\mu$-stationary system $(X,G,\nu)$ and a bounded measurable
function $F : X \to \C$, the function
$$
f(g) = \int F(gx)\, d \nu(x),
$$
is measurable bounded $\mu$-harmonic function on $G$.
%; i.e. $f$ satisfies the equation $\mu*f =f$.
When $F$ is continuous the corresponding $f$ is in addition left uniformly continuous on $G$.
The universal measurable Poisson $\mu$-boundary $(\Pi(G,\mu),\eta)$ is characterized by the requirement that
the correspondence $F \mapsto f$ is a bounded linear surjection from $L^\infty(\Pi(G,\mu),\eta)$ onto
the space of measurable bounded $\mu$-harmonic functions on $G$.

The next structure theorem is proved in \cite[Theorem 4.3]{FG-10}.

\begin{thm}[Structure theorem for $\mu$-systems]\label{4.3}
Let
$\Xcal=(X,\nu,G)$ be a $\mu$-system, then there exist canonically
defined $\mu$-systems $\Pi(\Xcal)=(M,P^*)$ and  $\Xcal^* = \Xcal \vee \Pi(\Xcal)$ 
%with $\Pi(\Xcal)$  $\mu$-proximal, and  $\Xcal^*$ standard, 
with $\Xcal^*$ standard, as in the following diagram:
\begin{equation*}
\xymatrix
{
& \Xcal^* = \Xcal \vee \Pi(\Xcal) \ar[dl]_{\pi} \ar[dr]^{\sig}  &  \\
\Xcal &  & \Pi(\Xcal).
}
\end{equation*}
Here $\pi$ is a $\mu$-proximal extension, and $\sig$ is a measure
preserving extension. Thus every $\mu$-system admits a $\mu$-proximal
extension which is standard. The $\mu$-system $\Xcal$ is measure
preserving iff $\Pi(\Xcal)$ is trivial. The $\mu$-system $\Xcal$
is $\mu$-proximal iff both $\pi$ and $\sig$ are isomorphisms. 
%Moreover this structure is unique; i.e. any two such diagrams are isomorphic.
%We call $\Xcal^*$ the {\bf standard cover} of $\Xcal$.
\end{thm}

\br

\subsection{The topological universal Poisson $\mu$-boundary}

If one wants to realize all the bounded left uniformly continuous $\mu$-harmonic functions on $G$ in a single compact space
this usually comes with a price, this compact space may no longer be metrizable.

Let $\Lcal$ denote the $C^*$-algebra of bounded  left uniformly continuous complex valued functions on $G$.
We let $(B, \nu, G)$ be the 
{\bf topological universal Poisson $\mu$-boundary} corresponding to the pair $(G,\mu)$,
as in \cite[Chapter V.4, pages 63--67]{Gl-76}. 
Thus $(B,G) = (B_\mu,G)$ is a {\bf flow}, $\nu$ is a $\mu$-stationary probability measure on $B$
%(i.e. $\mu *\nu = \nu$) 
and the operator $L_\nu : C(B) \to \Lcal$, defined by
\begin{equation}\label{harmonic}
L_\nu F (g) = \int F(gx) \, d\nu(x),
\end{equation}
is an isometric isomorphism from $C(B)$ onto $\Hcal$,
where $\Hcal \subset \Lcal$ is the closed subspace comprising the
$\mu$-harmonic functions (i.e. those functions $f \in \Lcal$ satisfying the equation
$\mu * f = f$).

%The group $G$ admits a  universal minimal strongly proximal flow,
%denoted by $(\Pi_s(G),G)$. It is unique in the sense that the identity map
%is the only endomorphism it admits (see \cite[Chapter III]{Gl-76}).

\br

\section{The main theorem}\label{sec:main}

Let $G$ be a locally compact second countable group and $\mu$ an admissible measure on $G$.
Let $(B, \nu) = (B_\mu,\nu)$ be the corresponding topological universal Poisson $\mu$-boundary.
From Theorem 4.4 in Chapter V of \cite{Gl-76} we have:

%\begin{thm}\label{factor}
%For every $\mu$-stationary flow $(X,\la, G)$ 
%(i.e. a flow $(X,G)$ with a distinguished $\mu$-stationary measure $\la$) there is a continuous homomorphism
%$\phi : B \to M(X)$ such that $\phi(B) \supset X$ and $\phi_*(\nu) = \la$.
%\end{thm}

\begin{thm}\label{factor}
Let  $(X,\la, G)$ be a $\mu$-stationary flow
(i.e. a flow $(X,G)$ with a distinguished $\mu$-stationary measure $\la$).
If the $\mu$-stationary measure $\la$ is contractible then there is a continuous homomorphism
$\phi : B \to M(X)$ such that $\phi(B) \supset X$ and $\phi_*(\nu) = \la$.
\end{thm}

And, the following:

\begin{cor}\label{cor}
The universal minimal strongly proximal flow $\Pi_s(G)$ is a factor of every minimal subset of
the universal Poisson $\mu$-boundary  $(B, \nu)$.
\end{cor}

\begin{proof}
As pointed out in Remark \ref{MK}, the flow $(\Pi_s(G),G)$ admits a $\mu$-stationary measure,
say $\la$. Since the flow $(\Pi_s(G), G)$ is minimal and strongly proximal it follows that $\la$ is
contractible.
Now apply Theorem \ref{factor} to the $\mu$-stationary system $(\Pi_s(G), \la,G)$.
\end{proof}

%
%\begin{thm}
%The universal minimal strongly proximal flow $\Pi_s(G)$ is a factor of every minimal subset of
%the universal Poisson $\mu$-boundary  $(B, \nu)$.
%\end{thm}
%
A (measure theoretical) $\mu$-stationary system $(X, \la, G)$ is {\em $\mu$-mean proximal} if it 
is $\mu$-proximal and it admits a topological 
model (still denoted $(X, \la, G)$) where $\la$ is the {\bf unique} $\mu$-stationary measure on $X$
(see \cite[Theorem 14.1] {Fu-73} and \cite[Theorem 8.5]{GW-16}).

\br

%The following theorem is a strengthening of \cite[Theorem 3.11]{HK-23}.
%The following theorem extends Theorem 3.11 of \cite{HK-23} to the case where $G$ is LCSC..

\begin{thm}\label{equiv}
The following conditions are equivalent:
\begin{enumerate}
\item
The measure $\nu$ is the unique $\mu$-stationary measure on $B$.
%\item
%The flow $(B, G)$ is minimal.
\item
The map $\phi : B \to \Pi_s(G)$ is a flow isomorphism.
\item
Every stationary $\mu$-proximal system $(X, \la, G)$ is $\mu$-mean proximal.
\end{enumerate}
\end{thm}

\begin{proof}
(1) $\imp$ (2):
As every minimal subflow of $B$ admits a $\mu$-stationary measure, our assumption implies
that $B$ contains a unique minimal subflow $M \subset B$ and $\nu$ is supported on $M$. 
It then follows from formula (\ref{harmonic}) that $M= B$.
%Let $S \subset B$ be the closed support of $\nu$. 
%As $\nu$ is quasi-invariant and $G$ is separable, it follows that $S$ is $G$-invariant.
%Let $M \subset B$ be a minimal subflow. If $S \not = M$, then
%$M$ will support a $\mu$-stationary measure other than $\nu$. 
%As this possibility  is ruled out by assumption, we conclude that $S = M$.
%But clearly also $B = S$, so that $S = M = B$.

By Theorem \ref{factor} we have a continuous homomorphism  $\phi : B \to \Pi_s(G)$.
%and by Corollary \ref{cor} 
%hence $\phi(M) = \Pi_s(G)$.
By 
%Theorem 4.1 in \cite{Gl-75} 
Theorem \ref{RIM}, applied to $\phi : B \to \Pi_s(G)$,
and the maximality of $\Pi_s(G)$ as a minimal strongly
proximal flow, the diagram (\ref{rim}) collapses and $\phi = \tilde{\phi}$
admits a RIM.
% (see \cite{Gl-75}).
%It then follows that the system $(B, \nu,G)$ is a standard system
%%, in the sense of \cite{FG-10}
%and has a structure $(B,\nu,G) \overset{\phi}{\to} (\Pi_s(G), \phi_*(\nu),G)$,
%with $\phi$ a measure preserving extension. 
Let $z \mapsto \la_z$ be a RIM, from $\Pi_s(G) \to M(B)$.
Let $\rho = \phi_*(\nu)$ and set
$$
\kappa = \int \la_z \, d \rho(z).
$$
Then $\kappa$ is a $\mu$-stationary measure on $B$ and the $\mu$-stationary
system $(B, \kappa,G)$ is standard with structure
$(B,\kappa,G) \overset{\phi}{\to} (\Pi_s(G), \rho),G)$.
Now, by assumption, $\kappa = \nu$.
As $(B,\nu,G)$ is $\mu$-proximal,
it follows from the structure theorem (Theorem \ref{4.3}) that $\phi$ is both measure 
preserving extension and a $\mu$-proximal extension,
whence, a measure isomorphism.
This means that $\rho$ a.s. $\la_z$ is a point mass. 
%(Theorem 4.3 in \cite{FG-10}). 
Finally, since a RIM is a continuous map
we conclude that $\phi$ is a topological isomorphism.

\br

(2) $\imp$ (1):
By assumption there is a $\mu$-stationary measure $\nu$ on $\Pi_s(G)$ such that
$(\Pi_s(G), \nu,G)$ is the universal Poisson $\mu$-boundary.
Suppose $\eta$ is another  $\mu$-stationary measure on $\Pi_s(G)$. Then,
by universality there is a homomorphism $\psi : (\Pi_s(G). \nu,G) \to (\Pi_s(G),\eta,G)$.
This map $\psi$ is then, in particular, an endomorphism of the flow $(\Pi_s(G),G)$.
However, the only such endomorphism is the identity map
and we conclude that $\eta = \phi_*(\nu) =\nu$.

\br

(3) $\iff$ (1): This is clear from the universality of $(B, \nu, G)$.
\end{proof}

%\begin{defn}
%We say that the group $G$ is  USB if,
%for some generating $\mu$,  the universal Poisson boundary is non  trivial and has a unique $\mu$-stationary measure.
%Equivalently, when $(B,\nu)$ is nontrivial and $B$ is minimal.
%Note that by Theorem \ref{equiv}, as in this case $B = \Pi_s(G)$, such a group is not amenable.
%Moreover, if $G$ acts effectively on $B$ then its amenable radical is trivial.
%\end{defn}

%\begin{defn}
%We say that the group $G$ is  USB if, 
%for some generating $\mu$,  it acts effectively on it
%universal Poisson $\mu$-boundary  $(B, \nu)$, and $\nu$ is its unique $\mu$-stationary measure.
%Equivalently, when $G$ acts effectively on $B$ and the flow $(B,G)$ is minimal.
%Note that by Theorem \ref{equiv}, as in this case $B = \Pi_s(G)$, such a group has a trivial
%amenable radical.
%\end{defn}

\br

\begin{rmk}
The first paragraph in the proof of the implication (1) $\imp$ (2) shows that
for any topological model $(X, \la,G)$ of a $\mu$-mean proximal system the flow $(X,G)$  is minimal.
Let $Q \subset M(X)$ be an irreducible affine sublow of $M(X)$.
By Theorem 2.3 in Chapter III.2 of \cite{Gl-76} the set $Y = \ol{\ext (Q)} \subset Q$ is a minimal strongly  proximal flow
and therefore it admits a $\mu$-stationary measure, say $\kappa$. 
Then, its barycenter, the probability measure
$\rho =  {\rm bary}(\kappa)$, is a stationary measure on $X$. By uniqueness $\rho = \la$.
Since the conditional measures of $\la$ are point masses, we conclude that $Q = M(X)$ and $Y = X$.
This briefly is the proof of Theorem 3.11  in \cite{HK-23}.
\end{rmk}

\br

In \cite{HK-23} the authors say:

\begin{quote}
Many natural examples of Poisson boundaries are in fact USB, namely, the Poisson boundary is being realized as a unique stationary measure on a compact space. The main tool for realizing the Poisson boundary on a compact space is the strip criterion of Kaimanovich \cite{K-00}, which proves, in many cases that the Poisson measure is actually unique. To name some examples (by no mean a complete list!) are linear groups acting on flag varieties \cite{BS-11}, \cite{K-00}, \cite{L-85}, hyperbolic groups acting on the Gromov boundary \cite{K-00}, non- elementary subgroups of mapping class groups acting on the Thurston boundary \cite{KM-96}, and non-elementary subgroups of 
{\rm Out}($F_n)$ 
acting on the boundary of the outer space \cite{H-16}. 
\end{quote}

\br

\begin{defn}
We say that the group $G$ has the  {\em Unique stationary boundary} property (USB for short) if, 
for some admissible $\mu$,  it acts effectively on its universal
Poisson $\mu$-boundary $(B,\nu)$ and $\nu$ is the unique $\mu$-stationary
probability mesure on $B$.
Note that by Theorem \ref{equiv}, as in this case $B = \Pi_s(G)$, such a group has a trivial
amenable radical (see \cite[Proposition 3.12]{HK-23}).
\end{defn}

\begin{rmk}
When $G$ is an infinite countable and discrete group the space $\Lcal = \ell^\infty(G)$ and then for
any admissible probability measure $\mu$ on $G$ the space $B$ is the 
 the Gelfand spectrum of left uniformly continuous bounded harmonic functions
which is a commutative von Neumann algebra.
As was shown by Ozawa it then follows that the flow $(B,G)$ can never be minimal.
Thus a USB group is not discrete. By Furstenberg's theory every semisimple Lie group with trivial center
is USB and then $B = G/P$. Are there USB groups where $B= B(G,\mu)$ is not a homogeneous space ?
\end{rmk}

%
%\begin{question}
%Characterize the class of countable discrete USB groups. 
%Are they necessarily $C^*$-simple ?
%\end{question}
%
%\begin{question}
%Is there an example of a LCSC group $G$ and an admissible measure $\mu$
%such that the corresponding topological universal Poisson $\mu$-boundary $(B,\nu)$ is minimal
%but $\nu$ is not the unique $\mu$-stationary measure on $B$ ?
%\end{question}


\begin{thebibliography}{WWW}


\bibitem{BS-11}
Sara Brofferio and Bruno Schapira, 
{\em Poisson boundary of $GL_d(\Q)$}, Israel J. Math. 185 (2011), 125--140.

\bibitem{Fu-63}
Harry Furstenberg, A Poisson formula for semi-simple Lie groups, Ann. of Math. (2) 77, (1963),
335--386.

\bibitem{Fu-71}
Harry Furstenberg,
{\em Random walks and discrete subgroups of Lie groups\/},
Advances in probability and related topics, Vol {\bfseries 1},
Dekkers, 1971, pp. 1--63.

\bibitem{Fu-73}
Harry Furstenberg,
{\em Boundary theory and stochastic processes on homogeneous spaces}, in
Furstenberg, Harry (1973), Calvin Moore (ed.), "Boundary theory and stochastic processes on homogeneous spaces", Proceedings of Symposia in Pure Mathematics, 26, AMS: 193--232

\bibitem{FG-10}
H. Furstenberg and E. Glasner,
{\em Stationary dynamical systems},
Dynamical numbers,  AMS,  
Contemporary Math. {\bf 532},  1--28,  Providence, Rhode Island, 2010.

\bibitem{Gl-75}
S. Glasner,
{\em Relatively invariant measures\/},
Pacific J. Math. Math.\ {\bfseries 58}, (1975), 393--410.

\bibitem{Gl-76}
S. Glasner,
{\em Proximal flows\/},
Lecture Notes in Math.\ {\bfseries 517}, Springer-Verlag, 1976.

\bibitem{Gl-03}
E. Glasner,
{\em Ergodic theory via joinings\/},
AMS, Surveys and Monographs, {\bf{101}}, 2003.

\bibitem{GW-16}
E. Glasner and B. Weiss,
{\em Weak mixing properties for non-singular actions}, 
Ergodic Theory Dynam. Systems 36
(2016), no. 7, 2203--2217.

\bibitem{HK-23}
Yair Hartman, and Mehrdad Kalantar,
{\em Stationary $C^*$-dynamical systems}
 (with an appendix by Uri Bader, Yair Hartman, and Mehrdad Kalantar),
 JEMS,  Vol. 25 (2023), No. 5, 1783--1821.
 
 \bibitem{H-16}
Camille Horbez, 
{\em The Poisson boundary of ${\rm Out}(F_N)$}, Duke Math. J. 165 (2016), no. 2, 341--369.
 
\bibitem{K-00}
 Vadim A. Kaimanovich, 
 {\em The Poisson formula for groups with hyperbolic properties},
  Ann. of Math. (2) 152 (2000), no. 3, 659--692.
  
 \bibitem{KM-96}
Vadim A. Kaimanovich and Howard Masur, 
{\em The Poisson boundary of the mapping class group},
 Invent. Math. 125 (1996), no. 2, 221--264.  
  
\bibitem{L-85}
Fran\c{c}ois Ledrappier, 
{\em Poisson boundaries of discrete groups of matrices},
 Israel J. Math. 50 (1985), no. 4, 319--336.
 
  \bibitem{NZ-99}
A. Nevo and R. J. Zimmer,
{\em Homogenous projective factors for
actions of semi-simple Lie groups\/},
Invent.\ Math.\ {\bfseries 138}, (1999), 229--252. 

\end{thebibliography}
\end{document}